\documentclass[reqno,12pt]{amsart}
\usepackage{amsmath,amsthm,amsfonts,amssymb}
\usepackage[utf8x]{inputenc}
\usepackage[english]{babel}
\usepackage[normalem]{ulem} 

\date{} 
\newtheorem{theorem}{Theorem}[section]

\newtheorem{corollary}[theorem]{Corollary}

\numberwithin{equation}{section}

\DeclareMathOperator{\supp}{supp}

\begin{document}

\title[]{Approximate Monge solutions continuously depending on the parameter}
\maketitle

\begin{center}
S.N. Popova~\footnote{Moscow Institute of Physics and Technology; National Research University Higher School of Economics.}
\end{center}

\vskip .2in

{\bf Abstract.}
We consider Kantorovich optimal transportation problem in the case where the cost function and marginal distributions continuously depend on a parameter with
values in a metric space.
%We study Kantorovich type optimal transportation problems depending on the parameter. 
We prove the existence of approximate optimal Monge mappings continuous with respect to the parameter. 


\vskip .2in

\section{Introduction}

We recall that, given two Borel probability measures $\mu$ and $\nu$ on topological spaces $X$ and $Y$ respectively 
and a nonnegative Borel function $h$ on~$X\times Y$, the Kantorovich optimal transportation problem concerns minimization of  the integral
$$
K_h(\mu, \nu) = \inf \Bigl\{\int h\, d\sigma: \sigma \in \Pi(\mu, \nu) \Bigr\}
$$
over all measures $\sigma$ in the set $\Pi(\mu,\nu)$ consisting of Borel probability measures on $X\times Y$
with projections $\mu$ and $\nu$ on the factors, that is, $\sigma (A\times Y)=\mu(A)$ and $\sigma (X\times B)=\nu(B)$
for all Borel sets $A\subset X$ and $B\subset Y$. The measures $\mu$ and $\nu$ are called marginal distributions
or marginals, and $h$ is called a cost function.
In general, there is only infimum $K_h(\mu,\nu)$, which may be infinite. If the cost function $h$ is continuous (or at least lower semicontinuous)  
and bounded and the measures $\mu$ and $\nu$ are Radon, then the minimum is attained and measures on which it is attained are called
optimal measures or optimal Kantorovich plans. The boundedness of $h$ can be replaced by the assumption that
there is a measure in $\Pi(\mu,\nu)$ with respect to which $h$ is integrable. 
The Monge problem for the same triple $(\mu, \nu, h)$ consists in finding a Borel mapping $T\colon X\to Y$ taking $\mu$ into $\nu$,
that is $\nu=\mu\circ T^{-1}$, $(\mu\circ T^{-1})(B)=\mu(T^{-1}(B))$ for all Borel sets $B \subset Y$, for which the integral
$$
M_h(\mu, \nu) = \inf \Bigl\{ \int h(x, T(x))\, \mu(dx): \mu\circ T^{-1} = \nu \Bigr\}
$$
is minimal. In general, there is only infimum $M_h(\mu, \nu)$ (possibly, infinite),
but in many interesting cases there exist optimal Monge mappings. In any case,
$K_h(\mu,\nu)\le M_h(\mu, \nu)$, but if both measures are Radon, $\mu$ has no atoms and is separable, and the cost function $h$ is continuous, 
then $K_h(\mu,\nu)= M_h(\mu, \nu)$ (see \cite{BKP}, \cite{P}). This equality implies that if there is a unique
solution $T$ to the Monge problem, then the image of $\mu$ under the mapping $x\mapsto (x,T(x))$ is an optimal
Kantorovich plan. 
General information about Monge and Kantorovich problems can be found
in \cite{AG}, \cite{BK}, \cite{RR}, \cite{Sant}, and~\cite{V2}.

We consider optimal transportation of measures on metric and topological spaces
in the case where the cost function $h_t$ and marginal distributions $\mu_t$ and $\nu_t$  depend on a parameter $t$ with
values in a metric space. Kantorovich problems depending on a parameter were investigated in \cite{V2}, \cite{Z}, \cite{Mal}, \cite{BM}, where the questions of measurability were studied. 
We address the problem of continuity with respect to the parameter. Here the questions naturally arise about the continuity with respect to $t$ of
the optimal cost $K_{h_t}(\mu_t,\nu_t)$ and also about the possibility to select an  optimal plan in $\Pi(\mu_t,\nu_t)$ continuous with respect to
the parameter.  
In \cite{BP22_arxiv}, \cite{BP22_dan} it was proved that the cost of optimal transportation is continuous  with respect to the parameter in the case of  continuous
dependence of the cost function and marginal distributions on this parameter. Furthermore, it was shown that it is not always possible to select an optimal plan continuously depending on the parameter $t$.
However, it is possible to select approximate optimal plans continuous with respect to the parameter. 
Continuous dependence on marginals was considered in \cite{Berg99}, \cite{SavZar14}, and~\cite{GS21}. 
Similar problems may be studied for nonlinear cost functionals (see \cite{GRST2}, \cite{B-VBP}, \cite{Backhoff}, \cite{BPR}, \cite{NKP}), see also the recent survey \cite{B-umn22}.

Introduce the notation and terminology that will be used in this paper. 
A nonnegative Radon measure on a topological space $X$ is a bounded Borel measure $\mu\ge 0$
such that for every Borel set $B$ and every $\varepsilon>0$ there is a compact set $K\subset B$ such that
$\mu(B\backslash K)<\varepsilon$ (see \cite{B07}). If $X$ is a complete separable metric space,
then all Borel measures are Radon.

The space $\mathcal{M}_r(X)$ of signed bounded Radon measures on $X$ can be equipped with
the weak topology generated by the seminorms
$$
\mu\mapsto \biggl|\int f\, d\mu\biggr|,
$$
where $f$ is a bounded continuous function.

A set $\mathcal{M}$ of nonnegative Radon measures on a  space $X$ is called uniformly tight,
if for every $\varepsilon>0$ there exists a  compact set $K\subset X$ such that
$\mu(X\backslash K)<\varepsilon$ for all $\mu\in\mathcal{M}$.

Let $(X, d_X)$ and $(Y, d_Y)$ be metric spaces. The space $X\times Y$ is equipped with the metric
$$
d((x_1,y_1),(x_2,y_2))= d_X(x_1,x_2)+d_Y(y_1,y_2).
$$
The weak topology on the spaces of Radon probability measures $\mathcal{P}_r(X)$, $\mathcal{P}_r(Y)$,  $\mathcal{P}_r(X\times Y)$
is metrizable by  the  corresponding Kantorovich--Rubinshtein metrics  $d_{KR}$ (also called the
Fortet--Mourier  metrics, see \cite{B18}) defined by
$$
d_{KR}(\mu,\nu)=\sup \biggl\{\int f\, d(\mu - \nu)\colon f\in {\rm Lip}_1, \ |f|\le 1\biggr\},
$$
where ${\rm Lip}_1$ is the space of $1$-Lipschitz functions. If $X$ is complete, then
$(\mathcal{P}_r(X), d_{KR})$ is also complete and if $X$ is Polish, then $\mathcal{P}_r(X)$ is also Polish.

In this paper we study the existence of approximate optimal Monge mappings continuous with respect to the parameter. 
Section 2 addresses the case where the measures $\mu \in \mathcal P_r(X)$ and $\nu \in \mathcal P_r(Y)$ are fixed and
$h \colon X \times Y \times T \to [0, \infty)$ is a continuous cost function. In Section 3 we assume that the measure $\mu \in \mathcal P_r(X)$ is fixed and the measures $\nu_t \in \mathcal P_r(Y)$ continuously depend on $t$ in the weak topology.  
We prove that there exist approximate Monge solutions $T_t^{\varepsilon}$ such that $T_t^{\varepsilon}$ is continuous in $t$
in the sense of convergence $\mu$-a.e.: if $t_n \to t$ as $n \to \infty$, 
then $T_{t_n}^{\varepsilon} \to T_t^{\varepsilon}$ $\mu$-a.e. We also generalize this result to the case where the measures $\mu_t$ are continuous in $t$ in the total variation norm and the measures $\nu_t$ are continuous in $t$ in the weak topology. 

\section{The Monge problem with fixed marginals}

In \cite{BP22_arxiv} the question was addressed whether it is possible to select an optimal plan continuously depending on the parameter $t$.
The examples were constructed which show that such a choice is not always possible.
However, the situation improves for approximate optimal plans. 
Given $\varepsilon>0$, a measure $\sigma\in \Pi(\mu,\nu)$ will be called
$\varepsilon$-optimal for the cost function $h$ if
$$
\int h\, d\sigma \le K_h(\mu,\nu)+\varepsilon.
$$

\begin{theorem}[\cite{BP22_arxiv}] \label{t-appr} 
Let $X$, $Y$ be complete metric spaces. Let $T$ be a metric space, and  for every $t\in T$
we are given measures $\mu_t\in \mathcal{P}_r(X)$ and $\nu_t\in \mathcal{P}_r(Y)$ such that
the mappings $t\mapsto \mu_t$ and $t\mapsto \nu_t$ are continuous in the weak topology
(which is equivalent to the continuity in the Kantorovich--Rubinshtein metric). Suppose also
that there is a continuous nonnegative function $(t,x,y)\mapsto h_t(x,y)$. 
Suppose that for every $t$ there exist nonnegative Borel functions \mbox{$a_t\in L^1(\mu_t)$} and $b_t\in L^1(\nu_t)$
such that 
\begin{equation}\label{mainboundt}
h_t(x,y)\le a_t(x)+b_t(y),
\quad
\lim\limits_{R\to+\infty} \sup_t \biggl(\int_{\{a_t\ge R\}} a_t\, d\mu_t+ \int_{\{b_t\ge R\}} b_t\, d\nu_t\biggr)=0.
\end{equation}
Then one can select $\varepsilon$-optimal  measures $\sigma_t^\varepsilon\in \Pi(\mu_t,\nu_t)$ for the cost functions $h_t$
such that they will be continuous in $t$ in the weak topology for every fixed $\varepsilon>0$.

If for every $t$ there is a unique optimal plan~$\sigma_t$, then it is continuous
in~$t$.
\end{theorem}

In this paper we strengthen the result from \cite{BP22_arxiv} looking at approximate optimal Monge mappings continuously depending on the parameter. 

First, we consider the particular case where the marginals $\mu \in \mathcal P_r(X)$, $\nu \in \mathcal P_r(Y)$ are fixed and cost functions $h_t$ depend on the parameter $t$. 
We prove the following result on the existence of approximate optimal Monge mappings continuously depending on the parameter $t$.   

\begin{theorem}\label{th1} 
Let $X, Y$ be completely regular topological spaces. Let $\mu$ be a non-atomic Radon probability measure on $X$, let $\nu$ be a Radon probability measure on $Y$, and the measures $\mu$ and $\nu$ are concentrated on countable unions of metrizable compact sets (i.e. we may assume that $X$ and $Y$ are Souslin spaces). Let $T$ be a metric space, $h \colon X \times Y \times T \to [0, \infty)$ be a continuous function such that $h(x, y, t) \le a_t(x) + b_t(y)$, where $a_t \in L^1(\mu)$, $b_t \in L^1(\nu)$ and
\begin{equation}\label{bound1}
\lim_{R \to +\infty} \sup_{t \in T} \Bigl(\int_{a_t \ge R} a_t d\mu + \int_{b_t \ge R} b_t d\nu \Bigr) = 0.
\end{equation}
Then for any $\varepsilon > 0$ one can select $\varepsilon$-optimal Monge mappings $T_t^{\varepsilon}$ for the cost functions $h_t$
such that $T_t^{\varepsilon}$ is continuous in $t$ in the sense of convergence $\mu$-a.e.: if $t_n \to t$ as $n \to \infty$, then $T_{t_n}^{\varepsilon} \to T_t^{\varepsilon}$ $\mu$-a.e.
\end{theorem}

\begin{proof}
We first consider the case where the function $h$ is bounded. We may assume that $h \le 1$.  
Let $\varepsilon > 0$. Set $\varepsilon_1 = \varepsilon/5$. Let us take a metrizable compact set $\tilde K_1 \subset X$ such that $\mu(X \setminus \tilde K_1) < \varepsilon_1/2$. 
Since the measure $\mu$ is non-atomic and the compact set $\tilde K_1$ is metrizable, 
the measure space $(\tilde K_1, \mu|_{\tilde K_1})$ is almost homeomorphic to 
$([0, \mu(\tilde K_1)], \lambda)$, where 
$\lambda$ is Lebesgue measure (see \cite[Theorem~9.6.3]{B07}). Let \linebreak $\varphi \colon [0, \mu(\tilde K_1)] \to \tilde K_1$ be an almost homeomorphism. Then there exists a compact set  
$S \subset [0, \mu(\tilde K_1)]$ such that
$0 < \lambda([0, \mu(\tilde K_1)] \setminus S) < \varepsilon_1/2$ and $\varphi|_{S}$ is a homeomorphism. 
Denote $K_1 = \varphi(S)$. Then $K_1$ is a metrizable compact set and the measure space $(K_1, \mu|_{K_1})$ is homeomorphic to $(S, \lambda)$. Moreover, we have 
$$
0 < \mu(X \setminus K_1) = \mu(X \setminus \tilde K_1) + \lambda([0, \mu(\tilde K_1)] \setminus S) < \varepsilon_1.
$$
Let us take a metrizable compact set $K_2 \subset Y$ such that $\nu(Y \setminus K_2) \le \mu(X \setminus K_1)$. 
Let $d_{K_1}$ be the metric generating the topology on $K_1$. 

Let us prove that there exists a continuous (strictly positive) function $\delta \colon T \to (0, +\infty)$ such that for any 
$x_1, x_2 \in K_1$, $y \in K_2$, $t \in T$ we have $|h(x_1, y, t) - h(x_2, y, t)| < \varepsilon_1$ if $d_{K_1}(x_1, x_2) < \delta(t)$. Since $h$ is continuous on $K_1 \times K_2 \times T$, it follows that for any $t_0 \in T$ there exists a real number
$\kappa_{t_0} > 0$ and an open neighbourhood $W_{t_0} \subset T$
($t_0 \in W_{t_0}$) such that $|h(x_1, y, t) - h(x_2, y, t)| < \varepsilon_1$ for any $x_1, x_2 \in K_1$ with $d_{K_1}(x_1, x_2) < \kappa_{t_0}$ and for any $y \in K_2$, $t \in W_{t_0}$. 
The metric space $T$ posseses a locally finite continuous partition of unity $\{\psi_{\alpha}, \alpha \in A\}$ subordinated to the open cover $\{W_t, t \in T\}$, i.e. a set of continuous functions $\psi_{\alpha}$, $\alpha \in A$, such that $0 \le \psi_{\alpha} \le 1$ for any $\alpha \in A$, $\supp \psi_{\alpha} \subset W_{\tau(\alpha)}$ for some $\tau(\alpha) \in T$, for every point $t \in T$ there exists a neighbourhood $W$ such that $W \cap \supp \psi_{\alpha} \neq \varnothing$ for at most finite number of indices $\alpha \in A$, and $\sum_{\alpha} \psi_{\alpha}(t) = 1$.

Set 
$$\delta(t) = \sum_{\alpha} \kappa_{\tau(\alpha)} \psi_{\alpha}(t).$$ 
Then the function $\delta(t)$ is continuous, since for any point $t \in T$ there exists a neighbourhood $W$ such that $\delta(t)$ is equal to the sum of a finite number of continuous functions on $W$. Let us show that the function $\delta(t)$ satisfies the required condition. Fix $t_0 \in T$. Let $\alpha_1, \dots, \alpha_N$ be all indices from the set $A$ such that $\psi_{\alpha_i}(t_0) \neq 0$. Then $t_0 \in W_{\tau(\alpha_i)}$ for all $i \in \{1, \dots, N\}$. 
The equality $\sum_{\alpha} \psi_{\alpha}(t_0) = 1$ implies that $0 < \delta(t_0) \le \max(\kappa_{\tau(\alpha_1)}, \dots, \kappa_{\tau(\alpha_N)})$. Therefore, by the definition of the numbers $\kappa_t$ we have  
$|h(x_1, y, t_0) - h(x_2, y, t_0)| < \varepsilon_1$ if  $x_1, x_2 \in K_1$, $d_{K_1}(x_1, x_2) < \delta(t_0)$, $y \in K_2$.

Let us build a partition $$S = \bigsqcup_{j = 1}^{\infty} S_j(t)$$ satisfying the following properties: 
\begin{itemize}
\item[1)] for any $j \in \mathbb N$ the mapping $t \mapsto I_{S_j(t)}$ (where $I_{B}$ denotes the indicator function of a set $B$) is continuous in the sense of convergence $\lambda$-a.e., that is, for any sequence $t_n \to t$, $n \to \infty$, we have 
$I_{S_j(t_n)} \to I_{S_j(t)}$ $\lambda$-a.e.,
\item[2)] for any $j \in \mathbb N$ and for any $t \in T$ we have
$|h(\varphi(s_1), y, t) - h(\varphi(s_2), y, t)| < \varepsilon_1$ for all $s_1, s_2 \in S_j(t)$, $y \in K_2$.
\end{itemize}

Since the mapping $\varphi$ is continuous, as proven above, there exists a continuous function $\tilde \delta \colon T \to (0, +\infty)$ such that for any $s_1, s_2 \in S$, $y \in K_2$, $t \in T$ we have $|h(\varphi(s_1), y, t) - h(\varphi(s_2), y, t)| < \varepsilon_1$ if $|s_1 - s_2| \le \tilde \delta(t)$. Set $$S_j(t) = S \cap [(j-1) \tilde \delta(t), j \tilde \delta(t)), \quad j \in \mathbb N. $$ 
Then $S = \bigsqcup_{j = 1}^{\infty} S_j(t)$. From the definition of the function $\tilde \delta(t)$ it follows that the property 2) is satisfied. Let us prove that the property 1) is fulfilled. Let $t_n \to t$ as $n \to \infty$. For any $j \in \mathbb N$ let us show that $I_{S_j(t_n)} \to 1$ for all $s \in S \cap ((j - 1) \tilde \delta(t), j \tilde \delta(t))$. 
Fix $s \in S$, $s \in ((j - 1) \tilde \delta(t), j \tilde \delta(t))$. 
Then for all sufficiently large numbers $n$ it holds that $s \in ((j - 1) \tilde \delta(t_n), j \tilde \delta(t_n))$, since  
$\tilde \delta(t_n) \to \tilde \delta(t)$. Therefore, $I_{S_j(t_n)}(s) = 1$ for all sufficiently large $n$. 
Thus for all $s \in S \cap ((j - 1) \tilde \delta(t), j \tilde \delta(t))$ and for all $i \in \mathbb N$ 
we have $I_{S_i(t_n)}(s) \to I_{S_i(t)}(s)$. Therefore, the property 1) is satisfied. 

Set $X_j(t) = \varphi(S_j(t))$. Then $K_1 = \bigsqcup_{j = 1}^{\infty} X_j(t)$. We have 
$I_{X_j(t_n)} \to I_{X_j(t)}$ $\mu$-a.e., if $t_n \to t$, $n \to \infty$ (this also implies that
$\mu(X_j(t_n) \triangle X_j(t)) \to 0$ as $n \to \infty$). Furthermore, for any $j \in \mathbb N$ and for any $t \in T$ we have $|h(x_1, y, t) - h(x_2, y, t)| < \varepsilon_1$ for all $x_1, x_2 \in X_j(t)$, $y \in K_2$.

Consider the Kantorovich problem with the cost function $h(x, y, t)$ and measures $\mu|_{K_1}$, $\alpha \nu|_{K_2}$, where $\alpha = \mu(K_1)/\nu(K_2) \le 1$.  
By Theorem \ref{t-appr} there exist $\varepsilon$-optimal measures $\pi_t \in \Pi(\mu|_{K_1}, \alpha \nu|_{K_2})$ for the cost function $h(x, y, t)$ such that $\pi_t$ is continuous in $t$ in the weak topology.
Let $\nu^j_t$ be the projection of the measure $I_{X_j(t)} \pi_t$ on $Y$, $j \in \mathbb N$. Let us show that $\nu^j_t$ is continuous in $t$ 
in the weak topology. Let $t_n \to t$ as $n \to \infty$, we show that the measures $\nu^j_{t_n}$ converge weakly to $\nu^j_t$. We have $\|I_{X_j(t_n)} \pi_{t_n} - I_{X_j(t)} \pi_{t_n}\| =    
\mu(X_j(t_n) \triangle X_j(t)) \to 0$, where $\|\cdot\|$ is the total variation norm. Therefore, it is sufficient to prove that the measures $I_{X_j(t)} \pi_{t_n}$ converge weakly to $I_{X_j(t)} \pi_t$. 
Let $g \in C_b(X \times Y)$, $|g| \le 1$, we show that
$$
\int_{X \times Y} g(x, y) I_{X_j(t)} \pi_{t_n}(dx dy) \to \int_{X \times Y} g(x, y) I_{X_j(t)} \pi_t(dx dy). 
$$
Fix $\delta > 0$. Take a compact set $F_j$ and an open set $U_j$ such that $F_j \subset X_j(t) \subset U_j$ and $\mu(U_j \setminus F_j) < \delta$. There exists a continuous function $f \colon X \to \mathbb R$ such that $f = 1$ on $F_j$, $f = 0$ outside $U_j$, $0 \le f \le 1$. Then 
$$\int_{X \times Y} f(x) g(x, y) \pi_{t_n}(dx dy) \to \int_{X \times Y} f(x) g(x, y) \pi_t(dx dy),$$ 
since $\pi_{t_n}$ converge weakly to $\pi_t$. 
Furthermore, we have $|I_{X_j(t)} - f(x)| \le I_{U_j \setminus F_j}$. Therefore, 
\begin{multline*}
\Bigl|\int_{X \times Y} (I_{X_j(t)} g(x, y) \pi_{t_n}(dx dy) - \int_{X \times Y}  f(x) g(x, y) \pi_{t_n}(dx dy) \Bigr| \le \\ \le 
\int_{X \times Y} |(I_{X_j(t)} - f(x)) g(x, y)| \pi_{t_n}(dx dy) \le 
\int_{X \times Y} I_{U_j \setminus F_j} \pi_{t_n}(dx dy) = \mu(U_j \setminus F_j) < \delta. 
\end{multline*}
From above we obtain  
\begin{multline*}
\Bigl|\int_{X \times Y} g(x, y) I_{X_j(t)} \pi_{t_n}(dx dy) - \int_{X \times Y} g(x, y) I_{X_j(t)} \pi_t(dx dy) \Bigr| \le  \\ \le
\Bigl|\int_{X \times Y} f(x) g(x, y) \pi_{t_n}(dx dy) - \int_{X \times Y} f(x) g(x, y) \pi_t(dx dy)\Bigr| + 2 \delta. 
\end{multline*}
Hence $\int g(x, y) I_{X_j(t)} \pi_{t_n}(dx dy) - \int g(x, y) I_{X_j(t)} \pi_t(dx dy) \to 0$. 
Therefore, the measures $\nu^j_{t_n}$ converge weakly to $\nu^j_t$, i.e. the mapping $t \mapsto \nu^j_t$ is continuous in the weak topology. 

Since the compact set $K_2$ is metrizable, it posseses the strong Skorohod property (see \cite{B18}), that is, for any probability measure 
$\eta$ on $K_2$ there exists a mapping $\xi_{\eta} \colon [0, 1] \to K_2$ such that 
$\lambda \circ \xi_{\eta}^{-1} = \eta$, where $\lambda$ is Lebesgue measure on $[0, 1]$, and if measures $\eta_n$ converge weakly to 
$\eta$, then $\xi_{\eta_n} \to \xi_{\eta}$ $\lambda$-a.e. 

Since the mapping $t \mapsto \nu^j_t$ is continuous in the weak topology for any $j \in \mathbb N$, by the strong Skorohod property 
for any $j \in \mathbb N$ there exists a mapping $\xi_{t, j} \colon [0, \lambda(S_j(t))] \to K_2$ such that
$$\lambda|_{[0, \lambda(S_j(t))]} \circ \xi_{t, j}^{-1} = \nu^j_t$$ 
and $\xi_{t, j}$ is continuous in $t$ in the sense of convergence $\lambda$-a.e. Set 
$$F^j_t(s) = \lambda([0, s] \cap S_j(t)). $$ 
Then the mapping $t \mapsto F^j_t$ is continuous in $t$ in the topology of pointwise convergence: if $t_n \to t$ as $n \to \infty$, then $F^j_{t_n}(s) \to F^j_t(s)$ for any $s \in S$. 
Indeed, $|F^j_{t_n}(s) - F^j_t(s)| \le \lambda(S_j(t_n) \triangle S_j(t)) \to 0$ as $n \to \infty$.  
Set 
$$T_t(x) = \xi_{t, j}(F^j_t(\varphi^{-1}(x))) \quad \mbox{if  } x \in X_j(t), \, j \in \mathbb N.
$$
Then $\mu|_{X_j(t)} \circ T_t^{-1} = \nu^j_t$, since $\varphi^{-1} \colon K_1 \to S$ is a homeomorpism which transfers  the measure $\mu|_{X_j(t)}$ to the measure $\lambda|_{S_j(t)}$ and the mapping $F^j_t$ transfers the measure 
$\lambda|_{S_j(t)}$ to the measure $\lambda|_{[0, \lambda(S_j(t))]}$. 
Therefore, $\mu|_{K_1} \circ T_t^{-1} = \alpha \nu|_{K_2}$. 
Since the measure $\mu$ is non-atomic, there exists a mapping $T \colon X \setminus K_1 \to Y$ such that 
$$\mu|_{X \setminus K_1} \circ T^{-1} = \nu - \alpha \nu|_{K_2}.$$
Set $T_t(x) = T(x)$ for any $x \in X \setminus K_1$. Then $\mu \circ T_t^{-1} = \nu$. 

Let us show that the mapping $T_t$ is continuous in $t$ in the sense of convergence $\mu$-a.e. Let $t_n \to t$, $n \to \infty$. Prove that for any $j \in \mathbb N$ 
$$\mu(\{x \in X_j(t): T_{t_n}(x) \not \to T_t(x)\}) = 0.$$ 
For $\mu$-a.e. $x \in X_j(t)$ it holds that $x \in X_j(t_n)$ for all sufficiently large $n$, 
since $I_{X_j(t_n)} \to I_{X_j(t)}$ $\mu$-a.e. Therefore, for $\mu$-a.e. $x \in X_j(t)$ we have for all sufficiently large $n$ 
$$T_{t_n}(x) = \xi_{t_n, j}(F^j_{t_n}(\varphi^{-1}(x))) \to \xi_{t, j}(F^j_t(\varphi^{-1}(x))) = T_t(x),$$  
since
$F^j_{t_n}(\varphi^{-1}(x)) \to F^j_t(\varphi^{-1}(x))$ due to continuity of $F^j_t$ in $t$ 
and $\xi_{t_n, j} \to \xi_{t, j}$ $\lambda$-a.e. 
Thus $\mu(\{x \in X: T_{t_n}(x) \not \to T_t(x)\}) = 0$ and the mapping $T_t$ is continuous in $t$ in the sense of convergence $\mu$-a.e.

Let us show that the mapping $T_t$ is $\varepsilon$-optimal for every $t \in T$. Fix $t \in T$. 
For any $j \in \mathbb N$ we have (fix some $x_0 \in X_j(t))$
\begin{multline*}
\Bigl|\int_{X_j(t)} h_t(x, T_t x) \mu(dx) - \int_{K_2} h_t(x_0, y) \nu^j_t(dy) \Bigr| =  \\  
= \Bigl|\int_{X_j(t)} (h_t(x, T_t x) - h_t(x_0, T_t x)) \mu(dx)\Bigr| < \varepsilon_1 \mu(X_j(t)), 
\end{multline*}
since $\mu|_{X_j(t)} \circ T_t^{-1} = \nu_t^j$ and 
$|h_t(x, y) - h_t(x_0, y)| < \varepsilon_1$ for any $x \in X_j(t)$, $y \in K_2$.
Similarly
\begin{multline*}
\Bigl|\int_{X_j(t) \times K_2} h_t(x, y) \pi_t(dx dy) - \int_{K_2} h_t(x_0, y) \nu^j_t(dy)\Bigr| = \\ 
= \Bigl|\int_{X_j(t) \times K_2} (h_t(x, y) - h_t(x_0, y)) \pi_t(dx dy)\Bigr| < \varepsilon_1 \mu(X_j(t)).
\end{multline*}
Therefore, 
$$\int_{X_j(t)} h_t(x, T_t x) \mu(dx) \le \int_{X_j(t) \times K_2} h_t(x, y) \pi_t(dx dy) + 2 \varepsilon_1 \mu(X_j(t)).$$ 
Summing over $j \in \mathbb N$, we obtain the inequality  
$$\int_{K_1} h_t(x, T_t x) \mu(dx) \le \int_{K_1 \times K_2} h_t(x, y) \pi_t(dx dy) + 2 \varepsilon_1. $$ 
Moreover, $\int_{X \setminus K_1} h_t(x, T_t x) \mu(dx) \le \mu(X \setminus K_1) < \varepsilon_1$. 
Hence 
$$\int_X h_t(x, T_t x) \mu(dx) \le \int_{K_1 \times K_2} h_t(x, y) \pi_t(dx dy) + 3 \varepsilon_1. $$

Let $\sigma \in \Pi(\mu, \nu)$ be an optimal measure in the Kantorovich problem with the cost function $h_t(x, y)$ and measures $\mu, \nu$. Let $\mu_1$ and $\nu_1$ be the projections of the measure $I_{K_1 \times K_2} \sigma$ on $X$ and $Y$ respectively. 
Set $\tilde \sigma = \alpha I_{K_1 \times K_2} \sigma + \zeta$, where 
$\zeta \in \Pi(\mu|_{K_1} - \alpha \mu_1, \alpha \nu|_{K_2} - \alpha \nu_1)$. 
Then $\tilde \sigma \in \Pi(\mu|_{K_1}, \alpha \nu|_{K_2})$ and hence 
\begin{multline*}
\int_{K_1 \times K_2} h_t(x, y) \pi_t(dx dy) \le \int_{K_1 \times K_2} h_t(x, y) \tilde \sigma(dx dy) + \varepsilon_1
\le \\ \le \int_{K_1 \times K_2} h_t(x, y) \sigma (dx dy) + (\nu(K_2) - \nu_1(K_2)) + \varepsilon_1.  
\end{multline*}
We have $\nu(K_2) - \nu_1(K_2) = \sigma((X \setminus K_1) \times K_2) \le \mu(X \setminus K_1) < \varepsilon_1$. 

Therefore,  
$$
\int_X h_t(x, T_t x) \mu(dx) \le \int_{K_1 \times K_2} h_t(x, y) \pi_t(dx dy) + 3 \varepsilon_1 \le  
\int_{X \times Y} h_t(x, y) \sigma(dx dy) + 5 \varepsilon_1. 
$$ 
So the mapping $T_t$ is $5\varepsilon_1$-optimal for any $t \in T$. 

Consider now the general case. Let $h(x, y, t) \le a_t(x) + b_t(y)$, where the functions $a_t \in L^1(\mu)$ and $b_t \in L^1(\nu)$ satisfy (\ref{bound1}). Let $N \in \mathbb N$. As proven above, for the bounded continuous function $\min(h, N)$ there exist 
$\varepsilon/2$-optimal Monge mappings $T_t$ which are continuous in $t$ in the sense of convergence $\mu$-a.e. 
For any measure $\sigma \in \Pi(\mu, \nu)$ we have
\begin{multline*}
\int h_t d \sigma - \int \min(h_t, N) d\sigma \le \int h_t I_{\{h_t \ge N\}} d\sigma \le \\ \le
\int (2 a_t I_{\{a_t \ge N/2\}} + 2 b_t I_{\{b_t \ge N/2\}}) d\sigma = 2 \int_{a_t \ge N/2} a_t d\mu + 2 \int_{b_t \ge N/2} b_t d\nu. 
\end{multline*}
Take $N \in \mathbb N$ such that $\int_{a_t \ge N/2} a_t d\mu + \int_{b_t \ge N/2} b_t d\nu < \varepsilon/4$. 
Then the mappings $T_t$ are $\varepsilon$-optimal for the cost function $h$.  
\end{proof}

\section{The Monge problem with marginals depending on the parameter}

Assume that the measure $\mu \in \mathcal P_r(X)$ is fixed and the measures $\nu_t \in \mathcal P_r(Y)$ continuously depend on $t$ in the weak topology. 
We show that one can select approximate optimal Monge mappings continuously depending on the parameter $t$ in the sense of convergence $\mu$-a.e.

\begin{theorem}\label{th2} 
Let $X, Y$ be complete metric spaces and let 
$\mu$ be a non-atomic Radon probability measure on $X$. 
Let $T$ be a metric space, the mapping $t \mapsto \nu_t$, $T \to \mathcal P_r(Y)$, is continuous in the weak topology, 
$h \colon X \times Y \times T \to [0, \infty)$ is a continuous function such that 
$h(x, y, t) \le a_t(x) + b_t(y)$, where $a_t \in L^1(\mu)$, $b_t \in L^1(\nu_t)$ and
$$
\lim_{R \to +\infty} \sup_{t \in T} \Bigl(\int_{a_t \ge R} a_t d\mu + \int_{b_t \ge R} b_t d\nu_t \Bigr) = 0. 
$$ 
Then for any $\varepsilon > 0$ one can select $\varepsilon$-optimal Monge mappings $T_t^{\varepsilon}$ for the cost functions $h_t$ and measures $\mu$, $\nu_t$ (i.e. $\mu \circ (T_t^{\varepsilon})^{-1} = \nu_t$ for every $t \in T$) such that
$T_t^{\varepsilon}$ is continuous in $t$ in the sense of convergence $\mu$-a.e.: if $t_n \to t$ as $n \to \infty$, then $T_{t_n}^{\varepsilon} \to T_t^{\varepsilon}$ $\mu$-a.e.
\end{theorem}

\begin{proof} 
The assertion of Theorem \ref{th2} reduces to the case where $h \le 1$. 
Let $\varepsilon > 0$. Set $\varepsilon_1 = \varepsilon/6$. 
Since the measure $\mu$ is non-atomic, there exists a compact set $K_1 \subset X$ such that $\mu(X \setminus K_1) < \varepsilon_1$ and $(K_1, \mu|_{K_1})$ is homeomorphic to $(S, \lambda)$, where $S \subset [0, 1]$ is a compact set and 
$\lambda$ is Lebesgue measure. Let $\varphi \colon S \to K_1$ be a homeomorphism, $\lambda|_S \circ \varphi^{-1} = \mu|_{K_1}$. Let $d_X$ and $d_Y$ be the metrics of $X$ and $Y$ respectively. 

Let us prove that there exists a continuous (strictly positive) function $\delta \colon T \to (0, +\infty)$ and a collection of closed sets $Y(t) \subset Y$, $t \in T$, such that for any $t \in T$ we have $\nu_t(Y \setminus Y(t)) < \varepsilon_1$ and
$|h(x_1, y, t) - h(x_2, y, t)| < \varepsilon_1$ for all $x_1, x_2 \in K_1$ with $d_X(x_1, x_2) < \delta(t)$ and for all $y \in Y(t)$. 

For any $t \in T$ take a compact set $K_2(t) \subset Y$ such that $\nu_t(Y \setminus K_2(t)) < \varepsilon_1$. Since $h$ is continuous on $K_1 \times Y \times T$, it follows that for any $t_0 \in T$ there exist real numbers $\kappa(t_0) > 0$, $r(t_0) > 0$ and an open neighbourhood $\tilde W_{t_0} \subset T$
($t_0 \in \tilde W_{t_0}$) such that $|h(x_1, y, t) - h(x_2, y, t)| < \varepsilon_1$ for any $x_1, x_2 \in K_1$ with $d_X(x_1, x_2) < \kappa(t_0)$ and for any $y \in K_2(t_0)^{r(t_0)}$ (where $B^r = \{y \in Y: d_Y(y, B) \le r\}$ is a closed $r$-neighbourhood of a set $B$ in the metric space $Y$), $t \in \tilde W_{t_0}$. Since the mapping $t \mapsto \nu_t$ is continuous in the weak topology and  
$\nu_{t_0}(Y \setminus K_2(t_0)) < \varepsilon_1$, there exists an oper neighbourhood $W'_{t_0} \subset T$ ($t_0 \in W'_{t_0}$) such that $\nu_t(Y \setminus K_2(t_0)^{r(t_0)}) < \varepsilon_1$ for any $t \in W'_{t_0}$. 
Set $W_{t_0} = \tilde W_{t_0} \cap W'_{t_0}$.  

The metric space $T$ posseses a locally finite continuous partition of unity \linebreak $\{\psi_{\alpha}, \alpha \in A\}$ subordinated to the open cover $\{W_t, t \in T\}$, i.e. a set of continuous functions $\psi_{\alpha}$, $\alpha \in A$, such that $0 \le \psi_{\alpha} \le 1$ for any $\alpha \in A$, $\supp \psi_{\alpha} \subset W_{\tau(\alpha)}$ for some $\tau(\alpha) \in T$, for every point $t \in T$ there exists a neighbourhood $W$ such that $W \cap \supp \psi_{\alpha} \neq \varnothing$ for at most finite number of indices $\alpha \in A$, and $\sum_{\alpha} \psi_{\alpha}(t) = 1$.

Set 
$$\delta(t) = \sum_{\alpha} \kappa(\tau(\alpha)) \psi_{\alpha}(t).$$ 
Then the function $\delta(t)$ is continuous, since for any point $t \in T$ there exists a neighbourhood $W$ such that $\delta(t)$ is equal to the sum of a finite number of continuous functions on $W$.
For any $t \in T$ choose an index $\alpha(t)$ from the finite set $\{\alpha \in A: \psi_{\alpha}(t) \neq 0\}$ for which
the value $\kappa(\tau(\alpha))$ is maximal. 
Set $$Y(t) = K_2(\tau(\alpha(t)))^{r(\tau(\alpha(t)))}.$$
Let us show that the function $\delta(t)$ and the sets $Y(t)$, $t \in T$, satisfy the required condition. 
Fix $t_0 \in T$. Let $\alpha_1, \dots, \alpha_N$ be all indices from the set $A$ such that $\psi_{\alpha_i}(t_0) \neq 0$. Then $t_0 \in W_{\tau(\alpha_i)}$ for all $i \in \{1, \dots, N\}$. 
Since $\sum_{\alpha} \psi_{\alpha}(t_0) = 1$, we have $\delta(t_0) \le \max(\kappa(\tau(\alpha_1)), \dots, \kappa(\tau(\alpha_N))) = \kappa(\tau(\alpha(t_0)))$. Therefore, by the definition of the numbers $\kappa(t)$ we obtain that 
$|h(x_1, y, t_0) - h(x_2, y, t_0)| < \varepsilon_1$ if $x_1, x_2 \in K_1$, $d_X(x_1, x_2) < \delta(t_0)$, $y \in Y(t_0)$. 
Moreover, $\nu_{t_0}(Y \setminus Y(t_0)) < \varepsilon_1$, because $t_0 \in W_{\tau(\alpha(t_0))}$. 

Since the mapping $\varphi$ is continuous, as proven above, there exists a continuous function $\tilde \delta \colon T \to (0, +\infty)$ and a collection of closed sets $Y(t) \subset Y$, $t \in T$, such that for any $t \in T$ we have $\nu_t(Y \setminus Y(t)) < \varepsilon_1$ and 
$|h(\varphi(s_1), y, t) - h(\varphi(s_2), y, t)| < \varepsilon_1$ for all $s_1, s_2 \in S$ with $|s_1 - s_2| \le \tilde \delta(t)$ and for all $y \in Y(t)$.

As described in the proof of Theorem \ref{th1}, we can construct a partition \linebreak $S~=~\bigsqcup_{j = 1}^{\infty} S_j(t)$ satisfying the following properties: 
\begin{itemize}
\item[1)] for any $j \in \mathbb N$ the mapping $t \mapsto I_{S_j(t)}$ is continuous in the sense of convergence $\lambda$-a.e., that is, for any sequence $t_n \to t$, $n \to \infty$, we have 
$I_{S_j(t_n)} \to I_{S_j(t)}$ $\lambda$-a.e.,
\item[2)] for any $j \in \mathbb N$ and for any $t \in T$ we have
$|h(\varphi(s_1), y, t) - h(\varphi(s_2), y, t)| < \varepsilon_1$ for all $s_1, s_2 \in S_j(t)$, $y \in Y(t)$.
\end{itemize}

Set $X_j(t) = \varphi(S_j(t))$. Then $K_1 = \bigsqcup_{j = 1}^{\infty} X_j(t)$. We have 
$I_{X_j(t_n)} \to I_{X_j(t)}$ $\mu$-a.e., if $t_n \to t$, $n \to \infty$ (this also implies that  
$\mu(X_j(t_n) \triangle X_j(t)) \to 0$ as $n \to \infty$). Furthermore, for any $j \in \mathbb N$ and for any $t \in T$ we have $|h(x_1, y, t) - h(x_2, y, t)| < \varepsilon_1$ for all $x_1, x_2 \in X_j(t)$, $y \in Y(t)$. 
Set $X_0(t) = X \setminus K_1$. 

By Theorem \ref{t-appr} there exist $\varepsilon_1$-optimal measures $\pi_t \in \Pi(\mu, \nu_t)$ for the cost function $h(x, y, t)$ such that $\pi_t$ is continuous in $t$ in the weak topology. Let $\nu^j_t$ be the projection of the measure $I_{X_j(t)} \pi_t$ on $Y$, $j \in \mathbb N \cup \{0\}$. Then $\nu^j_t$ is continuous in $t$ in the weak topology. Indeed, if $t_n \to t$ as $n \to \infty$, then the measures $\nu^j_{t_n}$ converge weakly to $\nu^j_t$, since
the measures $\pi_{t_n}$ converge weakly to $\pi_t$ and $\mu(X_j(t_n) \triangle X_j(t)) \to 0$. 

The complete metric space $Y$ posseses the strong Skorohod property for Radon measures (see \cite{B18}), 
that is, for any Radon probability measure $\eta$ on $Y$ there exists a mapping $\xi_{\eta} \colon [0, 1] \to Y$ such that
$\lambda \circ \xi_{\eta}^{-1} = \eta$, where $\lambda$ is Lebesgue measure on $[0, 1]$, and if measures $\eta_n$ converge weakly to $\eta$, then $\xi_{\eta_n} \to \xi_{\eta}$ $\lambda$-a.e.

Since the mapping $t \mapsto \nu^j_t$ is continuous in the weak topology for any $j \in \mathbb N \cup \{0\}$, by the strong Skorohod property for any $j \in \mathbb N \cup \{0\}$ there exists a mapping $\xi_{t, j} \colon [0, \mu(X_j(t))] \to Y$ (where $\mu(X_j(t)) = \lambda(S_j(t))$ for any $j \in \mathbb N$ and $\mu(X_0(t)) = \mu(X \setminus K_1)$) such that
$$\lambda|_{[0, \mu(X_j(t))]} \circ \xi_{t, j}^{-1} = \nu^j_t$$ 
and $\xi_{t, j}$ is continuous in $t$ in the sense of convergence $\lambda$-a.e. 
Let $$F^j_t(s) = \lambda([0, s] \cap S_j(t)), \quad j \in \mathbb N. $$ 
The mapping $t \mapsto F^j_t$ is continuous in $t$ in the topology of pointwise convergence: if $t_n \to t$ as $n \to \infty$, then
$F^j_{t_n}(s) \to F^j_t(s)$ for any $s \in S$. Indeed, $|F^j_{t_n}(s) - F^j_t(s)| \le \lambda(S_j(t_n) \triangle S_j(t)) \to 0$ as $n \to \infty$.  
Set
$$
T_t(x) = \xi_{t, j}(F^j_t(\varphi^{-1}(x))) \quad \mbox{if  } x \in X_j(t), j \in \mathbb N.
$$ 
Then $\mu|_{X_j(t)} \circ T_t^{-1} = \nu^j_t$, since $\varphi^{-1} \colon K_1 \to S$ is a homeomorphism which transfers the measure $\mu|_{X_j(t)}$ to the measure $\lambda|_{S_j(t)}$ and the mapping $F^j_t$ transfers
$\lambda|_{S_j(t)}$ to the measure $\lambda|_{[0, \lambda(S_j(t))]}$. 
Since the measure $\mu$ is non-atomic, there exists a mapping  
$F \colon X \setminus K_1 \to [0, \mu(X \setminus K_1)]$ such that 
$$\mu|_{X \setminus K_1} \circ F^{-1} = \lambda|_{[0, \mu(X \setminus K_1)]}.$$ 
Set $T_t(x) = \xi_{t, 0}(F(x))$ for any $x \in X \setminus K_1$. 
Then $\mu|_{X \setminus K_1} \circ T_t^{-1} = \nu_t^0$. 
Therefore, $\mu \circ T_t^{-1} = \nu_t$ for any $t \in T$. 

Let us show that the mapping $T_t$ is continuous in $t$ in the sense of convergence $\mu$-a.e. Let $t_n \to t$, $n \to \infty$. Prove that for any $j \in \mathbb N$ 
$$\mu(\{x \in X_j(t): T_{t_n}(x) \not \to T_t(x)\}) = 0. $$ 
Indeed, for $\mu$-a.e. $x \in X_j(t)$ it holds that $x \in X_j(t_n)$ for all sufficiently large $n$, 
since $I_{X_j(t_n)} \to I_{X_j(t)}$ $\mu$-a.e. Therefore, for $\mu$-a.e. $x \in X_j(t)$ for all sufficiently large $n$ we have
$$T_{t_n}(x) = \xi_{t_n, j}(F^j_{t_n}(\varphi^{-1}(x))) \to \xi_{t, j}(F^j_t(\varphi^{-1}(x))) = T_t(x), $$
since $F^j_{t_n}(\varphi^{-1}(x)) \to F^j_t(\varphi^{-1}(x))$ due to the continuity of $F^j_t$ in $t$
and $\xi_{t_n, j} \to \xi_{t, j}$ $\lambda$-a.e.
Moreover, 
$$\mu(\{x \in X \setminus K_1: T_{t_n}(x) \not \to T_t(x)\} = \lambda(\{s \in [0, \mu(X \setminus K_1)]: \xi_{t_n, 0}(s) \not \to \xi_{t, 0}(s)\}) = 0.$$
Therofore, $\mu(\{x \in X: T_{t_n}(x) \not \to T_t(x)\}) = 0$ and the mapping $T_t$ is continuous in $t$ in the sense of convergence $\mu$-a.e.

Let us prove that the mapping $T_t$ is $\varepsilon$-optimal for any $t \in T$. Fix $t \in T$. 
For any $j \in \mathbb N$ we have (fix some $x_0 \in X_j(t))$
\begin{multline*}
\Bigl|\int_{X_j(t)} h_t(x, T_t x) \mu(dx) - \int_Y h_t(x_0, y) \nu^j_t(dy) \Bigr| = \\ =
\Bigl|\int_{X_j(t)} (h_t(x, T_t x) - h_t(x_0, T_t x)) \mu(dx)\Bigr| < \varepsilon_1 \mu(X_j(t)) + \mu(X_j(t) \setminus T_t^{-1}(Y(t))), 
\end{multline*}
since $\mu|_{X_j(t)} \circ T_t^{-1} = \nu^j_t$ and 
$|h_t(x, y) - h_t(x_0, y)| < \varepsilon_1$ for any $x \in X_j(t)$, $y \in Y(t)$.
Similarly 
\begin{multline*}
\Bigl|\int_{X_j(t) \times Y} h_t(x, y) \pi_t(dx dy) - \int_Y h_t(x_0, y) \nu^j_t(dy)\Bigr| = \\ = 
\Bigl|\int_{X_j(t) \times Y} (h_t(x, y) - h_t(x_0, y)) \pi_t(dx dy)\Bigr| < \varepsilon_1 \mu(X_j(t)) + \pi_t(X_j(t) \times (Y \setminus Y(t))). 
\end{multline*}
Therefore, 
\begin{multline*}
\int_{X_j(t)} h_t(x, T_t x) \mu(dx) \le \int_{X_j(t) \times Y} h_t(x, y) \pi_t(dx dy) + 2 \varepsilon_1 \mu(X_j(t)) +  \\ 
+ \mu(X_j(t) \setminus T_t^{-1}(Y(t))) + \pi_t(X_j(t) \times (Y \setminus Y(t))).
\end{multline*}
Summing over $j \in \mathbb N$, we obtain the inequality 
\begin{multline*}
\int_{K_1} h_t(x, T_t x) \mu(dx) \le \int_{K_1 \times Y} h_t(x, y) \pi_t(dx dy) + 2 \varepsilon_1 + \mu(X \setminus T_t^{-1}(Y(t))) + 
\pi_t(X \times (Y \setminus Y(t))) =  \\ = \int_{K_1 \times Y} h_t(x, y) \pi_t(dx dy) + 2 \varepsilon_1 + 2 \nu_t(Y \setminus Y(t)) \le
\int_{K_1 \times Y} h_t(x, y) \pi_t(dx dy) + 4 \varepsilon_1.
\end{multline*}
Furthermore, 
$$
\int_{X \setminus K_1} h_t(x, T_t x) \mu(dx) \le \mu(X \setminus K_1) < \varepsilon_1. 
$$
Therefore, $$\int_{X} h_t(x, T_t x) \mu(dx) \le \int_{X \times Y} h_t(x, y) \pi_t(dx dy) + 5 \varepsilon_1.$$
Thus the mapping $T_t$ is $6\varepsilon_1$-optimal for every $t \in T$. 
\end{proof}

\begin{corollary} \label{cor1}
The statement of Theorem \ref{th2} holds true if we replace the condition that $X$ is a complete metric space by the condition that $X$ is a completely regular topological space and the measure $\mu$ is concentrated on a countable union of metrizable compact sets (i.e. we may assume that $X$ is a Souslin space).
\end{corollary}

\begin{proof} 
Following the proof of Theorem \ref{th2} we construct the sets $Y(t)$ and partitions 
$K_1 = \bigsqcup_{j = 1}^{\infty} X_j(t)$, $t \in T$. According to Theorem \ref{t-appr}, consider $\varepsilon_1$-optimal measures $\pi_t~\in~\Pi(\mu|_{K_1}, \mu(K_1) \nu)$ in the Kantorovich problem for the measures $\mu|_{K_1}$ and $\mu(K_1) \nu$ with the cost function $h(x, y, t)$ such that $\pi_t$ is continuous in $t$ in the weak topology. Set $\nu^j_t = I_{X_j(t)} \pi_t$ for any $j \in \mathbb N$. Then $\nu^j_t$ is continuous in $t$ in the weak topology.  
Define the mapping $T_t$ on $K_1$ in the same way as in the proof of Theorem \ref{th2}, then we have 
$\mu|_{K_1} \circ T_t^{-1} = \mu(K_1) \nu_t$. Take a mapping $F \colon X \setminus K_1 \to [0, \mu(X \setminus K_1)]$ such that $$\mu|_{X \setminus K_1} \circ F^{-1} = \lambda|_{[0, \mu(X \setminus K_1)]}.$$ 
Set $T_t(x) = \xi_{t}(F(x))$ for any $x \in X \setminus K_1$, where $\xi_t \colon [0, \mu(X \setminus K_1)] \to Y$, 
$$\lambda|_{[0, \mu(X \setminus K_1)]} \circ \xi_t^{-1} = (1 - \mu(K_1))\nu_t$$ and $\xi_t$ is continuous in $t$ in the sense of convergence $\lambda$-a.e. Then $\mu \circ T_t^{-1} = \nu_t$, $T_t$ is continuous in $t$ in the sense of convergence $\mu$-a.e.
and $T_t$ is $\varepsilon$-optimal for every $t \in T$. 
\end{proof}

Consider now the most general case where the measures $\mu_t \in \mathcal P_r(X)$ and $\nu_t \in \mathcal P_r(Y)$ continuously depend on $t$. 
Assuming that the measures $\mu_t$ are continuous in $t$ in the total variation norm we prove the existence 
of approximate optimal Monge mappings continuously depending on the parameter $t$ in the sense of convergence $\mu_t$-a.e.

\begin{theorem} \label{th3}
Let $X$ be a complete separable metric space and let $Y$ be a complete metric space. Let $T$ be a metric space,
the mapping $t \mapsto \nu_t$, $T \to \mathcal P_r(Y)$, is continuous in the weak topology, the mapping $t \mapsto \mu_t$, $T \to \mathcal P_r(X)$, is continuous in the total variation norm, and the measures $\mu_t$ are non-atomic for all $t \in T$. 
Let 
$h \colon X \times Y \times T \to [0, \infty)$ be a continuous function such that $h(x, y, t) \le a_t(x) + b_t(y)$, where $a_t \in L^1(\mu_t)$, $b_t \in L^1(\nu_t)$ and
$$
\lim_{R \to +\infty} \sup_{t \in T} \Bigl(\int_{a_t \ge R} a_t d\mu_t + \int_{b_t \ge R} b_t d\nu_t \Bigr) = 0. 
$$
Then for any $\varepsilon > 0$ one can select $\varepsilon$-optimal Monge mappings $T_t^{\varepsilon}$ for the cost functions $h_t$ and measures $\mu_t$, $\nu_t$ (i.e. $\mu_t \circ (T_t^{\varepsilon})^{-1} = \nu_t$ for every $t \in T$) such that
$T_t^{\varepsilon}$ is continuous in $t$ in the sense of convergence $\mu_t$-a.e.: if $t_n \to t$ as $n \to \infty$, then $T_{t_n}^{\varepsilon} \to T_t^{\varepsilon}$ $\mu_t$-a.e.
\end{theorem}

\begin{proof}
The assertion of Theorem \ref{th3} reduces to the case where $h \le 1$. 
Let $\varepsilon > 0$. Set $\varepsilon_1 = \varepsilon/7$. 
Since every complete separable metric space is homeomorphic to a $G_{\delta}$-set in $[0, 1]^{\infty}$ (see \cite{Engelking}), we may assume that  $X \subset [0, 1]^{\infty}$. 
The compact metrizable space $[0, 1]^{\infty}$ is a continuous image of the Cantor set $C$, i.e. there exists a surjective continuous mapping $f \colon C \to [0, 1]^{\infty}$. By measurable selection theorem (see \cite{B07}) there exists a Borel measurable mapping $g \colon [0, 1]^{\infty} \to C$ such that $f(g(x)) = x$ for all $x \in [0, 1]^{\infty}$.
Set $\gamma_t = \mu_t \circ g^{-1}$, $t \in T$. Then $\mu_t = \gamma_t \circ f^{-1}$ for every $t \in T$ and the measures $\gamma_t$ are non-atomic. Moreover, the mapping $t \mapsto \gamma_t$ is continuous in the total variation norm, since
$\|\gamma_t - \gamma_{\tau}\| = \|(\mu_t - \mu_{\tau}) \circ g^{-1}\| \le \|\mu_t - \mu_{\tau}\|$ for any $t, \tau \in T$.  
Set $S = g(X)$. Then $S$ is a Borel subset of $C$. 
Let $d_X$ and $d_Y$ be the metrics on $X$ and $Y$ respectively. 

Let us prove that there exists a continuous (strictly positive) function $\delta \colon T \to (0, +\infty)$ and a collection of compact sets $X(t) \subset X$ and closed sets $Y(t) \subset Y$, $t \in T$, such that for any $t \in T$ we have  
$\mu_t(X \setminus X(t)) < \varepsilon_1$, $\nu_t(Y \setminus Y(t)) < \varepsilon_1$ and 
$|h(x_1, y, t) - h(x_2, y, t)| < \varepsilon_1$ for any $x_1, x_2 \in X(t)$ with $d_X(x_1, x_2) < \delta(t)$ and for any 
$y \in Y(t)$.  

For every $t \in T$ take compact sets $K_1(t) \subset X$ and $K_2(t) \subset Y$ such that 
$\mu_t(X \setminus K_1(t)) < \varepsilon_1$ and $\nu_t(Y \setminus K_2(t)) < \varepsilon_1$. 
Since $h$ is continuous on $X \times Y \times T$, for any $t_0 \in T$ there exist real numbers $\kappa(t_0) > 0$, $r(t_0) > 0$ and an open neighbourhood $\tilde W_{t_0} \subset T$
($t_0 \in \tilde W_{t_0}$) such that $|h(x_1, y, t) - h(x_2, y, t)| < \varepsilon_1$ for any $x_1, x_2 \in K_1(t_0)$ with $d_X(x_1, x_2) < \kappa(t_0)$ and for any $y \in K_2(t_0)^{r(t_0)}$ (where $B^r = \{y \in Y: d_Y(y, B) \le r\}$ is a closed  $r$-neighbourhood of a set $B$ in the metric space $Y$), $t \in \tilde W_{t_0}$. Since the mapping $t \mapsto \nu_t$ is continuous in the weak topology and $\nu_{t_0}(Y \setminus K_2(t_0)) < \varepsilon_1$, there exists an open neighbourhood  $W'_{t_0} \subset T$ ($t_0 \in W'_{t_0}$) such that $\nu_t(Y \setminus K_2(t_0)^{r(t_0)}) < \varepsilon_1$ for any $t \in W'_{t_0}$. 
Since the mapping $t \mapsto \mu_t$ is continuous in the total variation norm, there exists an open neighbourhood $W''_{t_0} \subset T$ ($t_0 \in W''_{t_0}$) such that
$\mu_t(X \setminus K_1(t_0)) < \varepsilon_1$ for any $t \in W''_{t_0}$. 
Set $W_{t_0} = \tilde W_{t_0} \cap W'_{t_0} \cap W''_{t_0}$.  

The metric space $T$ posseses a locally finite continuous partition of unity \linebreak $\{\psi_{\alpha}, \alpha \in A\}$ subordinated to the open cover $\{W_t, t \in T\}$, i.e. a set of continuous functions $\psi_{\alpha}$, $\alpha \in A$, such that $0 \le \psi_{\alpha} \le 1$ for any $\alpha \in A$, $\supp \psi_{\alpha} \subset W_{\tau(\alpha)}$ for some $\tau(\alpha) \in T$, for every point $t \in T$ there exists a neighbourhood $W$ such that $W \cap \supp \psi_{\alpha} \neq \varnothing$ for at most finite number of indices $\alpha \in A$, and $\sum_{\alpha} \psi_{\alpha}(t) = 1$.

Set 
$$\delta(t) = \sum_{\alpha} \kappa(\tau(\alpha)) \psi_{\alpha}(t).$$ 
Then the function $\delta(t)$ is continuous, since for any point $t \in T$ there exists a neighbourhood $W$ such that $\delta(t)$ is equal to the sum of a finite number of continuous functions on $W$.
For any $t \in T$ choose an index $\alpha(t)$ from the finite set $\{\alpha \in A: \psi_{\alpha}(t) \neq 0\}$ for which
the value $\kappa(\tau(\alpha))$ is maximal. 
Set 
$$X(t) = K_1(\tau(\alpha(t))),  \quad Y(t) = K_2(\tau(\alpha(t)))^{r(\tau(\alpha(t)))}. $$
Let us show that the function $\delta(t)$ and the sets $X(t)$, $Y(t)$, $t \in T$, satisfy the required condition. 
Fix $t_0 \in T$. Let $\alpha_1, \dots, \alpha_N$ be all indices from the set $A$ such that $\psi_{\alpha_i}(t_0) \neq 0$. 
Then $t_0 \in W_{\tau(\alpha_i)}$ for all $i \in \{1, \dots, N\}$. 
Since $\sum_{\alpha} \psi_{\alpha}(t_0) = 1$, we have $\delta(t_0) \le \max(\kappa(\tau(\alpha_1)), \dots, \kappa(\tau(\alpha_N))) = \kappa(\tau(\alpha(t_0)))$. 
Therefore, by the definition of the numbers $\kappa(t)$ we obtain that 
$|h(x_1, y, t_0) - h(x_2, y, t_0)| < \varepsilon_1$ if $x_1, x_2 \in X(t_0)$, $d_X(x_1, x_2) < \delta(t_0)$, $y \in Y(t_0)$. 
Moreover, $\mu_{t_0}(X \setminus X(t_0)) < \varepsilon_1$ and $\nu_{t_0}(Y \setminus Y(t_0)) < \varepsilon_1$, 
because $t_0 \in W_{\tau(\alpha(t_0))}$.

Since the mapping $f$ is continuous, the function $h(f(s), y, t)$ is continuous on $S \times Y \times T$. As proven above,  
 there exists a continuous function $\tilde \delta \colon T \to (0, +\infty)$ and a collection of sets  
$S(t) \subset S$, $Y(t) \subset Y$, $t \in T$, such that for any $t \in T$ we have $\gamma_t(S \setminus S(t)) < \varepsilon_1$, $\nu_t(Y \setminus Y(t)) < \varepsilon_1$ and 
$|h(f(s_1), y, t) - h(f(s_2), y, t)| < \varepsilon_1$ for all $s_1, s_2 \in S(t)$ with $|s_1 - s_2| \le \tilde \delta(t)$ and for all $y \in Y(t)$.  

As described in the proof of Theorem \ref{th1}, we can construct a partition \linebreak $S~=~\bigsqcup_{j = 1}^{\infty} S_j(t)$ satisfying the following properties: 
\begin{itemize}
\item[1)] for any $j \in \mathbb N$ the mapping $t \mapsto I_{S_j(t)}$ is continuous in the sense of convergence $\gamma_t$-a.e., that is, for any sequence $t_n \to t$, $n \to \infty$, we have 
$I_{S_j(t_n)} \to I_{S_j(t)}$ $\gamma_t$-a.e.,
\item[2)] for any $j \in \mathbb N$ and for any $t \in T$ we have
$|h(f(s_1), y, t) - h(f(s_2), y, t)| < \varepsilon_1$ for all $s_1, s_2 \in S(t) \cap S_j(t)$, $y \in Y(t)$.
\end{itemize}

Set $X(t) = f(S(t))$ and $X_j(t) = f(S_j(t))$, $j \in \mathbb N$. Then $X = \bigsqcup_{j = 1}^{\infty} X_j(t)$. We have 
$I_{X_j(t_n)} \to I_{X_j(t)}$ $\mu_t$-a.e.., if $t_n \to t$, $n \to \infty$ (this also implies that  
$\mu_t(X_j(t_n) \triangle X_j(t)) \to 0$ as $n \to \infty$). Furthermore, for any $j \in \mathbb N$ and for any $t \in T$ we have $|h(x_1, y, t) - h(x_2, y, t)| < \varepsilon_1$ for all $x_1, x_2 \in X(t) \cap X_j(t)$, $y \in Y(t)$. 

By Theorem \ref{t-appr} there exist $\varepsilon_1$-optimal measures $\pi_t \in \Pi(\mu_t, \nu_t)$ for the cost function $h(x, y, t)$ such that $\pi_t$ is continuous in $t$ in the weak topology. Let $\nu^j_t$ be the projection of the measure $I_{X_j(t)} \pi_t$ on $Y$, $j \in \mathbb N$.
Let us show that $\nu^j_t$ is continuous in $t$ in the weak topology. Let $t_n \to t$ as $n \to \infty$, we show that the measures $\nu^j_{t_n}$ converge weakly to $\nu^j_t$. 
We have
$$
\|I_{X_j(t_n)} \pi_{t_n} - I_{X_j(t)} \pi_{t_n} \| = \mu_{t_n}(X_j(t_n) \triangle X_j(t)) \le 
\|\mu_{t_n} - \mu_t\| + \mu_t(X_j(t_n) \triangle X_j(t)) \to 0, 
$$
since the mapping $t \mapsto \mu_t$ is continuous in the total variation norm. 
Let us prove that the measures $I_{X_j(t)} \pi_{t_n}$ converge weakly to $I_{X_j(t)} \pi_t$. 
Let $\zeta \in C_b(X \times Y)$, $|\zeta| \le 1$, we show that 
$$\int_{X \times Y} \zeta(x, y) I_{X_j(t)} \pi_{t_n}(dx dy) \to \int_{X \times Y} \zeta(x, y) I_{X_j(t)} \pi_t(dx dy).$$
Fix $\delta > 0$. Take a compact set $F_j$ and an open set $U_j$ such that $F_j \subset X_j(t) \subset U_j$ and $\mu_t(U_j \setminus F_j) < \delta$. There exist a continuous function $\chi \colon X \to \mathbb R$ such that $\chi = 1$ on $F_j$, $\chi = 0$ outside $U_j$, $0 \le \chi \le 1$. 
Then 
$$
\int_{X \times Y} \zeta(x, y) \chi(x) \pi_{t_n}(dx dy) \to \int_{X \times Y} \zeta(x, y) \chi(x) \pi_t(dx dy),
$$ 
since the measures $\pi_{t_n}$ converge weakly to $\pi_t$. Furthermore, 
\begin{multline*}
\Bigl|\int_{X \times Y} \zeta(x, y) I_{X_j(t)} \pi_{t_n}(dx dy)  - \int_{X \times Y} \zeta(x, y) \chi(x) \pi_{t_n}(dx dy)\Bigr| \le \\ \le \int_{X \times Y} I_{U_j \setminus F_j} \pi_{t_n}(dx dy) = \mu_{t_n}(U_j \setminus F_j) \le \|\mu_{t_n} - \mu_t\| + \mu_t(U_j \setminus F_j), 
\end{multline*}
since $|I_{X_j(t)} - \chi| \le I_{U_j \setminus F_j}$ and $|\zeta| \le 1$.  
Therefore, 
\begin{multline*}
\Bigl|\int_{X \times Y} \zeta(x, y) I_{X_j(t)} \pi_{t_n}(dx dy) - \int_{X \times Y} \zeta(x, y) I_{X_j(t)} \pi_t(dx dy)\Bigr| 
\le \\ \le
\Bigl|\int_{X \times Y} \zeta(x, y) \chi(x) \pi_{t_n}(dx dy) - \int_{X \times Y} \zeta(x, y) \chi(x) \pi_t(dx dy)\Bigr| +  \|\mu_{t_n} - \mu_t\| + 2 \delta. 
\end{multline*}
Hence we obtain that $\int_{X \times Y} \zeta(x, y) I_{X_j(t)} \pi_{t_n}(dx dy) - \int_{X \times Y} \zeta(x, y) I_{X_j(t)} \pi_t(dx dy) \to 0$.
Therefore, the measures $\nu^j_{t_n}$ converge weakly to $\nu^j_t$, i.e. the mapping $t \mapsto \nu^j_t$ is continuous in $t$ in the weak topology. 

The complete metric space $Y$ posseses the strong Skorohod property for Radon measures,  
that is, for any Radon probability measure $\eta$ on $Y$ there exists a mapping $\xi_{\eta} \colon [0, 1] \to Y$ such that
$\lambda \circ \xi_{\eta}^{-1} = \eta$, where $\lambda$ is Lebesgue measure on $[0, 1]$, and if measures $\eta_n$ converge weakly to $\eta$, then $\xi_{\eta_n} \to \xi_{\eta}$ $\lambda$-a.e. 

Since the mapping $t \mapsto \nu^j_t$ is continuous in the weak topology for any $j \in \mathbb N$, by the strong Skorohod property for any $j \in \mathbb N$ there exists a mapping $\xi_{t, j} \colon [0, \mu_t(X_j(t))] \to Y$ (where $\mu_t(X_j(t)) = \gamma_t(S_j(t))$ for any $j \in \mathbb N$) such that
$$\lambda|_{[0, \mu_t(X_j(t))]} \circ \xi_{t, j}^{-1} = \nu^j_t$$ 
and $\xi_{t, j}$ is continuous in $t$ in the sense of convergence $\lambda$-a.e. 
Let $$F^j_t(s) = \gamma_t([0, s] \cap S_j(t)), \quad j \in \mathbb N. $$ 
The mapping $t \mapsto F^j_t$ is continuous in $t$ in the topology of pointwise convergence: if $t_n \to t$ as $n \to \infty$, then $F^j_{t_n}(s) \to F^j_t(s)$ for any $s \in S$.
Indeed, 
$$
|F^j_{t_n}(s) - F^j_t(s)| \le 
\|\gamma_{t_n} - \gamma_t\| + \gamma_t(S_j(t_n) \triangle S_j(t)) \to 0, \quad n \to \infty.
$$
Set
$$
T_t(x) = \xi_{t, j}(F^j_t(g(x))) \quad \mbox{if  } x \in X_j(t), j \in \mathbb N. 
$$
Then $\mu_t|_{X_j(t)} \circ T_t^{-1} = \nu^j_t$, since the mapping $g$ transfers the measure $\mu_t|_{X_j(t)}$ to the measure $\gamma_t|_{S_j(t)}$ and the mapping $F^j_t$ transfers the measure 
$\gamma_t|_{S_j(t)}$ to the measure $\lambda|_{[0, \mu_t(X_j(t))]}$. 
Therefore, $\mu_t \circ T_t^{-1} = \nu_t$ for any $t \in T$. 

Let us show that the mapping $T_t$ is continuous in $t$ in the sense of convergence $\mu_t$-a.e. 
Let $t_n \to t$, $n \to \infty$. Prove that for any $j \in \mathbb N$ 
$$\mu_t(\{x \in X_j(t): T_{t_n}(x) \not \to T_t(x)\}) = 0.$$ 
Indeed, for $\mu_t$-a.e. $x \in X_j(t)$ it holds that $x \in X_j(t_n)$ for all sufficiently large $n$, 
since $I_{X_j(t_n)} \to I_{X_j(t)}$ $\mu_t$-a.e. 
Therefore, for $\mu_t$-a.e. $x \in X_j(t)$ for all sufficiently large $n$ we have
$$T_{t_n}(x) = \xi_{t_n, j}(F^j_{t_n}(g(x))) \to \xi_{t, j}(F^j_t(g(x))) = T_t(x),$$ 
since $F^j_{t_n}(g(x)) \to F^j_t(g(x))$ due to the continuity of $F^j_t$ in $t$
and $\xi_{t_n, j} \to \xi_{t, j}$ $\lambda$-a.e.
Therofore, $\mu_t(\{x \in X: T_{t_n}(x) \not \to T_t(x)\}) = 0$ and the mapping $T_t$ is continuous in $t$ in the sense of convergence $\mu_t$-a.e.

Let us prove that the mapping $T_t$ is $\varepsilon$-optimal for any $t \in T$. Fix $t \in T$. 
For any $j \in \mathbb N$ we have (fix some $x_0 \in X_j(t) \cap X(t))$
\begin{multline*}
\Bigl|\int_{X_j(t)} h_t(x, T_t x) \mu_t(dx) - \int_Y h_t(x_0, y) \nu^j_t(dy) \Bigr| = 
\Bigl|\int_{X_j(t)} (h_t(x, T_t x) - h_t(x_0, T_t x)) \mu_t(dx)\Bigr| \le \\ \le
\Bigl|\int_{X_j(t) \cap X(t)} (h_t(x, T_t x) - h_t(x_0, T_t x)) \mu_t(dx)\Bigr| + \mu_t(X_j(t) \setminus X(t))
< \\ < \varepsilon_1 \mu_t(X_j(t)) + 
\mu_t(X_j(t) \setminus T_t^{-1}(Y(t))) + \mu_t(X_j(t) \setminus X(t)), 
\end{multline*}
since $\mu_t|_{X_j(t)} \circ T_t^{-1} = \nu^j_t$ and 
$|h_t(x, y) - h_t(x_0, y)| < \varepsilon_1$ for any $x \in X_j(t) \cap X(t)$, $y \in Y(t)$.
Similarly
\begin{multline*}
\Bigl|\int_{X_j(t) \times Y} h_t(x, y) \pi_t(dx dy) - \int_Y h_t(x_0, y) \nu^j_t(dy)\Bigr| = 
\Bigl|\int_{X_j(t) \times Y} (h_t(x, y) - h_t(x_0, y)) \pi_t(dx dy)\Bigr| < \\ < \varepsilon_1 \mu_t(X_j(t)) + \pi_t(X_j(t) \times (Y \setminus Y(t))) + \pi_t((X_j(t) \setminus X(t)) \times Y). 
\end{multline*}
Therefore, 
\begin{multline*}
\int_{X_j(t)} h_t(x, T_t x) \mu_t(dx) \le \int_{X_j(t) \times Y} h_t(x, y) \pi_t(dx dy) + 2 \varepsilon_1 \mu_t(X_j(t)) +  \\ 
+ \mu_t(X_j(t) \setminus T_t^{-1}(Y(t))) + \pi_t(X_j(t) \times (Y \setminus Y(t))) + 2 \mu_t(X_j(t) \setminus X(t)). 
\end{multline*}
Summing over $j \in \mathbb N$, we obtain that 
\begin{multline*}
\int_{X} h_t(x, T_t x) \mu_t(dx) \le \int_{X \times Y} h_t(x, y) \pi_t(dx dy) + 2 \varepsilon_1 +  2 \mu_t(X \setminus X(t)) + \\
 + \mu_t(X \setminus T_t^{-1}(Y(t))) + \pi_t(X \times (Y \setminus Y(t))) 
= \int_{X \times Y} h_t(x, y) \pi_t(dx dy) + \\ + 2 \varepsilon_1 + 2 \mu_t(X \setminus X(t)) + 2 \nu_t(Y \setminus Y(t)) \le
\int_{X \times Y} h_t(x, y) \pi_t(dx dy) + 6 \varepsilon_1.
\end{multline*}
Therefore, the mapping $T_t$ is $7\varepsilon_1$-optimal for every $t \in T$. 
\end{proof}

\begin{corollary} \label{cor2}
The statement of Theorem \ref{th3} holds true in the case where $X$ is a Souslin space. 
\end{corollary}

\begin{proof} 
The Souslin space $X$ is an image of a complete separable metric space $\tilde X$ under a continuous surjective mapping 
$f \colon \tilde X \to X$. By measurable selection theorem (see \cite{B07}) there exists a mapping $g \colon X \to \tilde X$ such that $g$ is measurable with respect to the $\sigma$-algebra generated by Souslin sets and $f(g(x)) = x$ for all $x \in X$. 
Set $\gamma_t = \mu_t \circ g^{-1}$ for any $t \in T$. Then $\mu_t = \gamma_t \circ f^{-1}$ and the measures 
$\gamma_t$ are non-atomic. The mapping $t \mapsto \gamma_t$ is continuous in the total variation norm, 
since $\|\gamma_t - \gamma_{\tau}\| = \|\mu_t - \mu_{\tau}\|$ for any $t, \tau \in T$. 
The function $h(f(\tilde x), y, t)$ is continuous on $\tilde X \times Y \times T$. 
Consider the Kantorovich problem with the cost function $h(f(\tilde x), y, t)$ and measures 
$\gamma_t$, $\nu_t$, $t \in T$. 
By Theorem \ref{th3} there exist $\varepsilon$-optimal mappings $\tilde T_t \colon \tilde X \to Y$
such that $\tilde T_t$ is continuous in $t$ in the sense of convergence $\gamma_t$-a.e. Set $T_t(x) = \tilde T_t(g(x))$. 
Then $\mu_t \circ T_t^{-1} = \gamma_t \circ \tilde T_t^{-1} = \nu_t$ for any $t \in T$. The mapping $t \mapsto T_t$ is continuous in $t$ in the sense of convergence $\mu_t$-a.e. Indeed, if $t_n \to t$, $n \to \infty$, then 
$$\mu_t(\{x \in X: T_{t_n} x \not \to T_t x\} = \gamma_t(\{\tilde x \in \tilde X: \tilde T_{t_n} \tilde x \not \to \tilde T_t \tilde x\}) = 0.$$ 
Let us show that the mapping $T_t$ is $\varepsilon$-optimal for any $t \in T$. 
We have 
$$\int_X h(x, T_t x) \mu_t(dx) = \int_{\tilde X} h(f(\tilde x), \tilde T_t \tilde x) \gamma_t(d \tilde x).$$ 
Let $\sigma \in \Pi(\mu_t, \nu_t)$ be an optimal plan in the Kantorovich problem with the cost function $h(x, y, t)$ and measures $\mu_t, \nu_t$. 
Let $\tilde \sigma$ be the image of the measure $\sigma$ under the mapping $(x, y) \mapsto (g(x), y)$. Then $\tilde \sigma \in \Pi(\gamma_t, \nu_t)$ and 
$$\int_{\tilde X \times Y} h(f(\tilde x), y, t) \tilde \sigma(d \tilde x dy) = \int_{X \times Y} h(x, y, t) \sigma(dx dy). $$
Therefore, the minimum in the Kantorovich problem with the cost function $h(f(\tilde x), y, t)$ and measures $\gamma_t, \nu_t$ equals  the minimum in the Kantorovich problem with the cost function $h(x, y, t)$ and measures $\mu_t, \nu_t$. 
Therefore, the mapping $T_t$ is $\varepsilon$-optimal.
\end{proof}


\begin{thebibliography}{99}

\bibitem{AG}
L. Ambrosio, N. Gigli,
A user's guide to optimal transport,
Lecture Notes in Math.  2062 (2013), 1--155.

%\bibitem{AP}
%L. Ambrosio,  A. Pratelli,
%Existence and stability results in the $L^1$
%theory of optimal transportation,
%In: Optimal transportation and applications (Martina Franca, 2001),
%Lecture Notes in Math., V. 1813,  pp.~123--160, Springer, 2003.

%\bibitem{AF}
%J.-P. Aubin,  H. Frankowska,
%Set-valued Analysis,  Birkh\"auser Boston, Boston, 1990.

\bibitem{B-VBP}
J. Backhoff-Veraguas, M. Beiglb{\"o}ck, G. Pammer,
Existence, duality, and cyclical monotonicity for weak transport costs,
Calc. Var. Partial Differ. Equ. 58 (2019), Paper no. 203, pp.~1--28.

\bibitem{Backhoff}
J. Backhoff-Veraguas, G. Pammer,
Applications of weak transport theory, Bernoulli 28 (1) (2022), 370--394.

\bibitem{Berg99}
J. Bergin,
On the continuity of correspondences on sets of measures with restricted marginals.
Econom. Theory 13 (2) (1999), 471--481.

\bibitem{B07}
 V.I. Bogachev,
Measure Theory, vols.~1,~2, Springer, Berlin, 2007.

\bibitem{B18}
V.I. Bogachev, Weak Convergence of Measures, Amer. Math. Soc., Providence,
Rhode Island, 2018.

\bibitem{B22}
  V.I. Bogachev,
''Kantorovich problems with a parameter and density constraints'', Siber. Math. J. 63:1
 (2022), 34--47.

 \bibitem{B-umn22}
  V.I. Bogachev,
  ``The Kantorovich problem of optimal transportation of measures: 
new directions of research'', Uspehi Matem. Nauk 77:5 (2022),  3--52 (in Russian).


\bibitem{BKP}
V.I. Bogachev,  A.N. Kalinin, S.N. Popova,  On the equality of values in the Monge and Kantorovich problems,
  Zap. Nauchn. Sem. S.-Peterburg. Otdel. Mat. Inst. Steklov. (POMI) 457 (2017), 53--73 (Russian);
 translation in J. Math. Sci. (N.Y.) 238 (4) (2019), 377--389.

\bibitem{BK}
 V.I. Bogachev, A.V. Kolesnikov,
 The Monge--Kantorovich problem: achievements, connections, and prospects,
Uspekhi Matem. Nauk 67 (5) (2012), 3--110 (in Russian); English transl.:
Russian Math. Surveys 67 (5) (2012), 785--890.

\bibitem{BM}
V.I. Bogachev, I.I. Malofeev,
 Kantorovich problems and conditional measures depending on a parameter,
 J. Math. Anal. Appl. 486 (1) (2020), 1--30.

\bibitem{BP22_arxiv}
  V.I. Bogachev, S.N. Popova,
''Optimal transportation of measures with a parameter'',
arXiv:2111.13014v1.

\bibitem{BP22_dan}
  V.I. Bogachev, S.N. Popova, ''On Kantorovich problems with a parameter'', 
Dokl. Akad. Nauk 507 (1) (2022), 26--28. 


\bibitem{BPR}
  V.I. Bogachev, S.N. Popova, A.V. Rezbaev,
''On nonlinear Kantorovich problems with density constraints'', Moscow Mathematical Journal, (2022).

\bibitem{Engelking}
Engelking P. General topology. Polish Sci. Publ., Warszawa, 1977.

\bibitem{GS21}
   M.  Ghossoub, D. Saunders,
On the continuity of the feasible set mapping in optimal transport.
Econ. Theory Bull. 9 (1) (2021), 113--117.

\bibitem{GRST2}
   N. Gozlan, C. Roberto,  P.-M. Samson, P. Tetali,
  Kantorovich duality for general transport costs and applications,
  J. Funct. Anal. 273 (11) (2017), 3327--3405.


 \bibitem{Mal}
 I.I. Malofeev,  Measurable dependence of conditional measures on a parameter,
 Dokl. Akad. Nauk 470 (1) (2016), 13--17 (in Russian);
 English transl.: Dokl. Math. 94 (2) (2016), 493--497.

\bibitem{NKP}
S.N. Popova, ''On nonlinear Kantorovich problems for cost functions of a special form'', 
arXiv:2212.10473. 

 \bibitem{P}
A. Pratelli,
On the equality between Monge's infimum and Kantorovich's
 minimum in optimal mass transportation,
Ann. Inst. H.~Poincar\'e~(B), Probab. Statist. 43 (1) (2007), 1--13.

  \bibitem{RR}
 S.T. Rachev, L. R{\"u}schendorf,
 Mass Transportation Problems, vols.~I,~II, Springer, New York, 1998.

%\bibitem{RS}
% D. Repov\v{s}, P.V. Semenov,
%Continuous Selections and Multivalued Mappings,
%Kluwer Acad. Publ., Dordrech -- Boston -- London, 1998.

 \bibitem{Sant}
 Santambrogio F.
 Optimal Transport for Applied Mathematicians, Birk\-h\"au\-ser/\-Springer, Cham,
 2015.

 \bibitem{SavZar14}
 A. Savchenko, M. Zarichnyi,  Correspondences of probability measures with restricted marginals. Proc.
Intern. Geom. Center 7 (4)  (2014), 34--39.

\bibitem{V2}
C. Villani,  Optimal Transport, Old and New, Springer, New York, 2009.

\bibitem{Z}
X. Zhang,  Stochastic Monge--Kantorovich problem and its duality,
           Stochastics 85 (1) (2013), 71--84.

\end{thebibliography}
\end{document}